\documentclass[12pt]{amsart}

\usepackage{graphicx,latexsym,amsfonts,amssymb,txfonts,amsmath,amsthm}
\usepackage{pdfsync,color,tabularx,rotating}
\usepackage[all,cmtip]{xy}
\usepackage{url}
\usepackage{tikz}
\usepackage{epsfig}
\usepackage{comment}

\advance\hoffset by -0.9 truecm

\newtheorem{theo}{Theorem}[section]

\newtheorem{lemm}[theo]{Lemma}
\newtheorem{conj}[theo]{Conjecture}

\theoremstyle{remark}
\newtheorem{rema}[theo]{\bf Remark}

\theoremstyle{remark}
\newtheorem{exam}[theo]{\bf Example}

\theoremstyle{remark}
\newtheorem{defi}[theo]{\bf Definition}

\theoremstyle{remark}
\newtheorem{nota}[theo]{\bf Notation}

\theoremstyle{remark}
\newtheorem{ques}[theo]{\bf Question}

\def\tO{\widetilde{O}}

\begin{document}

\title{Primes in geometric series and finite permutation groups}

\author{Gareth A. Jones and Alexander K. Zvonkin}

\address{School of Mathematical Sciences, University of Southampton, Southampton SO17 1BJ, UK}
\email{G.A.Jones@maths.soton.ac.uk}

\address{LaBRI, Universit\'e de Bordeaux, 351 Cours de la Lib\'eration, F-33405 Talence 
Cedex, France}
\email{zvonkin@labri.fr}

\subjclass[2010]{11A41, 11N05, 11N32, 20B05, 20B25}.

\keywords{Permutation group, prime degree, projective space, Bunyakovsky conjecture, Goormaghtigh conjecture.}

\begin{abstract}
As a consequence of the classification of finite simple groups, the classification of 
permutation groups of prime degree is complete, apart from the question of when the 
natural degree $(q^n-1)/(q-1)$ of ${\rm L}_n(q)$ is prime. We present heuristic arguments 
and computational evidence to support a conjecture that for each prime $n\ge 3$ there are 
infinitely many primes of this form, even if one restricts to prime values 
of $q$.
\end{abstract}

\maketitle


\vspace{-1cm}

\section{Introduction}\label{Intro}

The study of transitive permutation groups of prime degree goes back to the work of Galois on 
polynomials of prime degree. It is sometimes asserted that the groups of prime degree are now 
completely known, as a consequence of the classification of finite simple groups. This assertion 
is true only if one ignores an apparently difficult number-theoretic problem, namely the existence 
or otherwise of infinitely many primes of a particular form. The list of such permutation groups 
includes several easily described infinite families, three relatively small sporadic examples, 
and one other family which will be the subject of this note.

Let $p$ be a prime, and let $q=p^e$, $e\ge 1$, be a prime power. The projective special 
linear groups ${\rm L}_n(q)={\rm PSL}_n(q)$ and some closely related groups act doubly transitively, 
with degree
\begin{equation}\label{projeqn}
m = \frac{q^n-1}{q-1}=1+q+q^2+\cdots+q^{n-1},
\end{equation}
on the points or hyperplanes of the projective space ${\mathbb P}^{n-1}({\mathbb F}_q)$ 
for integers $n\ge 2$ and prime powers $q\ge 2$.

\begin{defi}[Projective prime]
If the number $m$ of points and of hyperplanes of the projective space 
${\mathbb P}^{n-1}({\mathbb F}_q)$, defined by (\ref{projeqn}), is prime, 
then we call it a {\em projective prime}.
\end{defi}

\begin{rema}[$n$ prime]\label{rem:n-prime}
A necessary condition for $m$ in (\ref{projeqn}) to be prime is the primality of the
exponent $n$ since otherwise the polynomial $1+t+t^2+\cdots+t^{n-1}$
would be reducible over $\mathbb Z$.
\end{rema}

The only projective primes with $n=2$ are the Fermat primes $m=2^{2^k}+1$, 
while those with $q=2$ are the Mersenne primes, of the form $m=2^n-1$ with $n$ prime. 
However, there are many others, such as $m=13$ with $n=q=3$. An interesting case is 
the Mersenne prime
$$
m=31=1+2+4+8+16=1+5+25.
$$ 
Of course, it is an open problem whether there are infinitely many Fermat or Mersenne primes; 
at the time of writing, only five Fermat primes (with $k=0,\ldots, 4$) and $51$ Mersenne 
primes are known to exist. More generally, the existence of infinitely many projective 
primes seems to be an open problem. 

As in the case of Mersenne primes, there is plausible heuristic evidence, given in Section~\ref{Heuristic}, to support a conjecture that there are infinitely many projective primes. Indeed, there is much stronger computational 
support for this, even in restricted cases such as when $n=3$ and $q$ is prime (see Section~\ref{Comp}). Our aim in 
this note is to put forward such evidence, in the hope of inspiring specialists in number 
theory to address this problem. Thus, we formulate the following

\begin{conj}[Projective primes]
There are infinitely many projective primes.
\end{conj}

\begin{rema}[Addition of primes]
More than once, some mathematicians (and physicists), among them
very eminent ones, have expressed the opinion that the two famous Goldbach conjectures are completely devoid of 
interest since they concern the addition of primes whereas primes are created to be multiplied, 
not added. Our note may be considered as a strong case for the interest, in certain contexts, 
of the addition of primes and of numbers related to them (such as prime powers). By the way, one of the Goldbach conjectures has already been proved by Harald Helfgott~\cite{Helfgott}.
\end{rema}

\begin{nota}[Primes]
Notation for primes depends on the context. We follow the tradition of denoting a prime number 
by the letter $p$ when this number is treated alone. If, however, there are more than one prime 
number involved, as is the case, for example, in formula (\ref{projeqn}), then one of these 
numbers may be denoted by $m$ or by some other letter.
\end{nota}


\section{Transitive permutation groups of prime degree}\label{Groups}

This section summarises the background in finite permutation groups from which the problem 
stated in the Introduction arises. Any reader who is interested only in the problem itself 
can safely omit this section. For more details on permutation groups of prime degree, 
see~\cite[\S3.5]{DM} or \cite[\S V.21]{Hup}.

An elementary but important fact about transitive groups of prime degree is that they 
are all primitive, that is, they leave invariant no non-trivial equivalence relations. 
In particular, this means that rational or meromorphic functions of prime degree cannot 
be compositions of those of lower degree. Groups of prime degree are also rather rare: 
for example, there are $2\,801\,324$ transitive groups of degree $32$ (all but seven of 
them imprimitive), and only twelve groups of degree~$31$; similarly there are $315\,842$ 
groups of degree $40$, but only ten of degree $41$, and six of degree $47$.
The number of transitive groups of degree 32 is computed in \cite{Hulpke}. The database 
\cite{Galois-DB} contains a list of all transitive groups of degree $m\le 47$, $m\ne 32$. 
The GAP system \cite{GAP} contains a list of all {\sl primitive}\/ groups of degrees 
$m\le 2499$, and therefore, in particular, a list of all the groups of prime degrees 
up to the same limit.

One of the sections of the memoir by \'Evariste Galois \cite{Galois}\footnote{In his preface
of 16 January 1831, Galois writes that this text is an ``extrait d'un ouvrage que j'ai 
eu l'honneur de pr\'esenter \`a l'Acad\'emie il y a un an''. The French word ``ouvrage''
means either a book or just a large piece of work. This larger text was sent to Fourier for 
refereeing. Fourier had suddenly died, and the manuscript was never found among his papers.} 
is called ``Application to irreducible equations of prime degree''. If we translate the work 
of Galois on polynomials and their roots into modern terminology, he showed that the 
solvable groups of prime degree $p$ are the subgroups $G$ of the $1$-dimensional 
affine group
\[{\rm AGL}_1(p)=\{t\mapsto at+b\mid a, b\in{\mathbb F}_p, a\ne 0\}\]
containing the translation subgroup ${\rm C}_p=\{t\mapsto t+b\mid b\in{\mathbb F}_p\}$. Such 
groups $G$ are semidirect products $G={\rm C}_p\rtimes{\rm C}_d$ where ${\rm C}_d$
acts as a subgroup of~${\mathbb F}_p^*$ for some divisor $d$ of $p-1$. There is one 
group $G$ for each such $d$, including $G={\rm C}_p$ for $d=1$ and $G={\rm AGL}_1(p)$ 
for $d=p-1$.

This directs our attention to the nonsolvable groups of prime degree.
Burnside~\cite[\S251]{Bur11} showed that any such group $G$ must be doubly transitive
(as is ${\rm AGL}_1(p)$, unlike its proper subgroups). In fact, in this case
elementary arguments show that a minimal normal subgroup $S$ of $G$ is a nonabelian 
simple group, which is also transitive of degree $p$, with trivial centraliser $C_G(S)$ 
in $G$. Thus $G$ acts faithfully by conjugation on $S$, so
\[S\le G\le{\rm Aut}\,S.\]
This reduces the problem to that of determining the nonabelian simple groups $S$ of
prime degree~$p$, and then studying their automorphism groups for possible subgroups
$G$ of degree~$p$ (the action of~$S$ need not extend to all subgroups of ${\rm Aut}\,S$). 

The classification of finite simple groups was announced around 1980, though not 
completely proved until over twenty years later. One consequence (see~\cite{Cam81}, 
for example) was the classification of doubly transitive finite permutation groups. 
There are eight families, described 
in some detail in~\cite[\S7.7]{DM} and summarised in~\cite[\S7.4]{Cam}\footnote{Note that 
${\rm Aut}\,{\rm M}_{22}$, of degree $22$, is omitted from~\cite[p.~252]{DM}. Similarly, 
${\rm L}_2(11)$, of degree~$11$, is omitted from the list of groups of prime degree 
in~\cite[\S V.21.2]{Hup}, though it is mentioned in II.8.28(6) and in the 
Errata in the 2nd printing.}. As far as our problem is concerned, most of them can be 
ignored, as their degrees are composite: for example, the symplectic groups 
${\rm Sp}_{2n}(2)$ have degrees $2^{n-1}(2^n\pm 1)$, while the unitary and `small' 
Ree groups over ${\mathbb F}_q$ have degree $q^3+1$, divisible by $q+1$. 
The groups which survive this elimination process are listed in the following theorem:

\begin{theo}[Transitive groups of prime degree]\label{th:groups}
The nonabelian simple permutation groups $S$ of prime degree, together with any 
transitive groups $G\le{\rm Aut}\,S$ of degree equal to that of $S$, 
are as follows:
\begin{enumerate}
\item[(a)] 	alternating groups $S={\rm A}_p$ for primes $p\ge 5$, together with the 
			corresponding symmetric groups ${\rm Aut}({\rm A}_p)={\rm S}_p$; 
	
\item[(b)]  $G=S={\rm L}_2(11)$ for $p=11$, acting on the cosets of a subgroup ${\rm A}_5$ 
			{\rm (}two representations, on two conjugacy classes of such subgroups, 
			equivalent under ${\rm Aut}\,S={\rm PGL}_2(11)${\rm )}, and the Mathieu 
			groups $S={\rm M}_{11}$ and ${\rm M}_{23}$, acting on Steiner systems with 
			$p=11$ and $p=23$ points;
\item[(c)]  groups $G$ such that 	
			$S={\rm L}_n(q)\le G\le {\rm P\Gamma L}_n(q)\le {\rm Aut}\,({\rm L}_n(q))$,		
			acting on the points or hyperplanes of the projective space 
			${\mathbb P}^{n-1}({\mathbb F}_q)$ when the degree $m=(q^n-1)/(q-1)$ 
			is a projective prime, $m\ge 5$.
\end{enumerate}
If we also include the affine groups, where
\begin{enumerate}
\item[(d)]  ${\rm C}_p\le G\le {\rm AGL}_1(p)$ for primes $p$,
\end{enumerate}
then we have a complete list of the transitive groups of prime degree.
\end{theo}

\begin{rema}[Commentaries on Theorem \ref{th:groups}]\ 
\begin{enumerate}
\item	The group ${\rm L}_2(11)$ in (b) is one of three cases, known already to Galois, 
		in which the simple group ${\rm L}_2(p)$ has a non-trivial transitive representation 
		of degree less that $p+1$, specifically of degree $p=5,7$ or $11$ on the cosets of a 
		subgroup isomorphic to ${\rm A}_4$, ${\rm S}_4$ or ${\rm A}_5$. The first case appears 
		in both (a) and (c), via the isomorphisms 
		${\rm L}_2(5)\cong {\rm A}_5\cong {\rm L}_2(4)$, while the second appears in (c) 
		via the isomorphism ${\rm L}_2(7)\cong {\rm L}_3(2)$. 
		The group ${\rm Aut}\,({\rm L}_2(11))={\rm PGL}_2(11)$ does not have a
		representation of degree $11$; hence, only the group ${\rm L}_2(11)$ is 
		a member of our list.
\item	For the two Mathieu groups in (b) we have ${\rm Aut}\,({\rm M}_{11})={\rm M}_{11}$
		and ${\rm Aut}\,({\rm M}_{23})={\rm M}_{23}$.		
\item 	In (c), we have ${\rm Aut}\,({\rm L}_n(q))={\rm P\Gamma L}_n(q)$ if $n=2$, 
		but if $n\ge 3$ then ${\rm Aut}\,({\rm L}_n(q))$ contains ${\rm P\Gamma L}_n(q)$ 
		with index~$2$, the `extra automorphism' arising from the point-hyperplane duality 
		of ${\mathbb P}^{n-1}({\mathbb F}_q)$.
\item	In (d), the group ${\rm AGL}_1(p)$ is {\em not}\/ the automorphism group of
		${\rm C}_p$ (indeed, ${\rm Aut}\,({\rm C}_p)\cong{\rm C}_{p-1}$). The case (d)
		does not correspond to the general scheme of the cases (a), (b), (c) since,
		as explained above, the group ${\rm AGL}_1(p)$ is solvable.
\end{enumerate}
\end{rema}

For a given projective prime $m$, the groups $G$ in (c) are easily determined: 
they correspond bijectively to the subgroups of
\[{\rm P\Gamma L}_n(q)/{\rm L}_n(q)\cong({\rm PGL}_n(q)/{\rm L}_n(q))\rtimes 
{\rm Gal}\,{\mathbb F}_q\cong{\rm C}_d\rtimes{\rm C}_e\]
where $d=\gcd(q-1,n)$ and $q=p^e$ for some prime $p$. In fact, if $m$ is prime then 
$n$ is prime (see Remark \ref{rem:n-prime}) and $q\not\equiv 1$ mod~$(n)$, 
so $d=1$, the groups ${\rm PGL}_n(q)$ and ${\rm L}_n(q)$ coincide, and 
${\rm P\Gamma L}_n(q)/{\rm L}_n(q)\cong{\rm C}_e$. The real problem is to know which 
primes are projective, and thus correspond to groups in (c), and in particular whether 
or not there are infinitely many of them. 

Although this paper concentrates on those cases where ${\rm L}_n(q)$ has prime degree, 
there is also interest in cases such as ${\rm L}_5(3)$ where its natural degree $m$ is 
a prime power ($11^2$ in this case). For example, Guralnick~\cite{Gur} has shown that 
if a nonabelian simple group $S$ has a transitive representation of prime power degree, 
then $S$ is an alternating group or ${\rm L}_n(q)$ acting naturally, or ${\rm L}_2(11)$, 
${\rm M}_{11}$ or ${\rm M}_{23}$ acting as in Theorem~\ref{th:groups}(b), or the unitary 
group ${\rm U}_4(2)\cong {\rm Sp}_4(3)\cong {\rm O}_5(3)$ permuting the $27$ lines on 
a cubic surface. In particular, $S$ is doubly transitive in all cases except the last, 
where it has rank~3, that is, three orbits on ordered pairs. See also~\cite{EGSS},
where Estes, Guralnick, Schacher and Straus have shown that for each prime $p$
there are only finitely many $e, q, n\ge 3$ such that $p^e=(q^n-1)/(q-1)$.

Another related topic which we will not address here is the Feit--Thompson
Conjecture~\cite{FT62} (see also~\cite[Problem B25]{Guy}), that if $p$ and $q$ are 
distinct primes then $(p^q-1)/(p-1)$ does not divide $(q^p-1)/(q-1)$. A proof of this 
would significantly shorten the (very long) proof of the theorem~\cite{FT63} that groups 
of odd order are solvable.


\section{The Bunyakovsky Conjecture}\label{BunConj}

Viktor Bunyakovsky (1804--1889) was a Russian mathematician and a disciple of Cauchy. 
In Russia he is mainly known for the Cauchy--Bunyakovsky inequality
which, in the Western tradition, is named after Cauchy--Schwarz. (As is stated in the
Wikipedia, Bunyakovsky ``\ldots\ is credited with an early discovery of the Cauchy--Schwarz 
inequality, proving it for the infinite dimensional case in 1859, many years prior to Hermann Schwarz's works on the subject.")

In 1857, Bunyakovsky formulated the following conjecture (see \cite{Bun-1857, Bun-wiki}).

\begin{conj}[Bunyakovsky Conjecture]
The following fairly obvious necessary conditions for a polynomial $f(t)\in{\mathbb Z}[t]$ 
to have infinitely many prime values for $t\in{\mathbb N}$ are also sufficient:
\begin{itemize}
\item the leading coefficient of $f$ should be positive,
\item $f$ should be irreducible,
\item the integers $f(t)$ for $t\in{\mathbb N}$ should have greatest common divisor $1$.
\end{itemize}
\end{conj}

The last condition is needed in order to avoid examples such as $f(t)=t^2+t+2$, which 
satisfies the first two conditions but has even values for all $t\in{\mathbb N}$; 
Bunyakovsky gives the surprising example $f(t)=t^9-t^3+2520$, which is irreducible 
but has all its values divisible by $504$. His conjecture is a special case of
Schinzel's Hypothesis H~\cite{SS}, which concerns finite sets of polynomials simultaneously 
taking prime values.

\begin{rema}[Verification of the coprimality of $f(t)$ for $t\in\mathbb{N}$]
\label{rem:verif-Bun}
The existence of examples like the one above leads to the following question: how to 
verify that the greatest common divisor of $f(t)$ for $t\in{\mathbb N}$ is $1$? A~method
(today we would say, an algorithm) proposed by Bunyakovsky is based on the following
observations.
\begin{enumerate}
\item	Let $f(t)=c_nt^n+\cdots+c_1t+c_0$. If a prime $p$ divides all the values of
		$f(t)$ for $t\in{\mathbb N}$, then $p$ is a divisor of $c_0$. Indeed, substituting 
		$t=p$ in $f(t)$ and taking the result modulo $p$ we get $c_0 \equiv 0 \mod (p)$. 
		Thus, we have only a finite number of primes $p$ to test.
\item	Let $h(t)$ be a polynomial of degree $k<p$. Then all the values of 
		$h(t)$, $t\in{\mathbb N}$, are divisible by~$p$ if and only if all the
		coefficients of $h$ are divisible by~$p$. Indeed, otherwise, reducing $h(t)$
		modulo $p$ we would get a non-zero polynomial of degree less that $p$ which 
		would have $p$~roots.
\item	All the values of the polynomial $t^p-t$ are obviously divisible by $p$.
		Let $h(t)$ be the remainder of $f(t)$ on division by $t^p-t$. All that
		remains is to determine whether the coefficients of $h(t)$ are all divisible by $p$.
\end{enumerate}
\end{rema}

The conjecture is true for $\deg(f)=1$: this is Dirichlet's Theorem on primes in an 
arithmetic progression (see~\cite[\S 5.3.2]{BS} for a proof). However, it has not been 
proved for any polynomial of degree greater than $1$, including the case $f(t)=t^2+1$
(see~\cite[\S A1]{Guy},~\cite[\S 2.8]{HW} or~\cite[Ch.~3.IVD]{Rib}); this is sometimes 
called Landau's problem, though in fact it goes back to Euler~\cite{Eul}.
In our case we have the advantage 
that we are not restricted to a single polynomial: we may consider polynomials 
$f(t)=1+t+t^2+\cdots+t^{n-1}$ for any prime $n\ge 3$. (Since we have nothing to add 
to the current state of knowledge or ignorance concerning Fermat primes, we will 
assume for the rest of this paper that $n\ne 2$.) On the other hand, we require 
prime values of $f(t)$ where $t$ is a {\sl prime power\/}, so a proof of the Bunyakovsky 
Conjecture for such a polynomial would not necessarily yield infinitely many projective 
primes.

Finally, we note that the Bunyakovsky Conjecture has recently arisen in a similar way 
in the construction by Amarra, Devillers and Praeger~\cite{ADP} of block-transitive
point-imprimitive $2$-designs with specific parameters.


\section{Heuristic arguments}\label{Heuristic}

In this section we will present some heuristic arguments to support the conjecture that 
there are infinitely many projective primes $m=(q^n-1)/(q-1)$. They are based on heuristic 
arguments used elsewhere in considering the distribution and number of primes of a given 
form. In particular, some of the arguments in this section are adapted from Wagstaff's 
treatment~\cite{Wag} of conjectures of Gillies, Lenstra and Pomerance
about Mersenne primes, and its summary in 
Prime Pages\footnote{See \url{https://primes.utm.edu/mersenne/heuristic.html}.}. 
Of course heuristic arguments, based on assumptions which, although plausible, cannot be 
rigorously justified, do not prove anything (in particular, see the warning in Section~\ref{warning}).
However, they may suggest results which one 
could attempt to prove by more legitimate means. Authors of classic texts did not disdain 
such kind of arguments: see, for example, Sections 2.5 and 22.20 of the book \cite{HW} by Hardy and 
Wright, where they present heuristic evidence that there are only finitely many Fermat primes
whereas there are infinitely many prime pairs;
see also the discussion of probabilistic methods in~\cite[Notes on Ch.~8.3]{NZM}
and P\'olya's carefully-qualified defence of heuristic reasoning in number theory 
in~\cite{Pol}.

Beside the ``general'' conjecture of infinitely many projective primes we also formulate
a number of ``specific'' (and therefore stronger) conjectures concerning projective
primes of some specific forms. Their plausibility is based mainly on a series of
computational results presented in Sections~\ref{Comp} and \ref{sec:stochastic}.


\subsection{Prime divisors of $m$}\label{sec:small}
We consider firstly the case of any fixed prime $n\ge 3$, and secondly that of any fixed 
prime power $q$. In each case, we will need the following lemma in order to give better 
estimates for the number of projective primes up to some bound.

\begin{lemm}[Prime divisors of $m$]\label{SmallPrimes}
Let $m=(q^n-1)/(q-1)$ for some integer $q$ and prime $n\ge 3$, and let $r$ be a prime 
dividing $m$. Then either $r\equiv 1$ {\rm mod}~$(2n)$ 
{\rm (}so in particular $r\ge 2n+1${\rm )},
or $r=n$ with $q\equiv 1$ {\rm mod}~$(n)$.
Conversely, if $q\equiv 1$ {\rm mod}~$(n)$ then $m$ is divisible by $n$.
\end{lemm}

\noindent{\sl Proof.} If a prime $r$ divides $m$ then $q^n\equiv 1$ mod~$(r)$. 
Since $n$ is prime, it follows that either $n$ divides the order $r-1$ of the 
multiplicative group ${\mathbb F}_r^*$, or $q\equiv 1$ mod~$(r)$.

If $n$ divides $r-1$ then $r\equiv 1$ mod~$(n)$. Clearly $m=1+q+\cdots +q^{n-1}$ is odd, 
and hence so is $r$, so $r\equiv 1$ mod~$(2n)$ since $n$ is odd and hence $r\ge 2n+1$. 

If $q\equiv 1$ mod~$(r)$ then 
\[m=1+q+\cdots+q^{n-1}\equiv\underbrace{1+1+\cdots+1}_{n\,\,{\rm times}}\equiv n \; 
{\rm mod}~(r).\]
However, $m\equiv 0$ mod~$(r)$, so $n\equiv 0$ mod~$(r)$ and hence $r=n$ since 
$n$ is prime. The converse is obvious. 
\hfill$\square$

\medskip

Thus $m$ is not divisible by any prime $r\le 2n$, except the prime $r=n$ if 
$q\equiv 1$ mod~$(n)$.

\begin{exam}[For Lemma \ref{SmallPrimes}]\label{ex:small-primes}
Let $n=3$. If $q=11\not\equiv 1$ mod~$(3)$ then $m=133=7\cdot 19$; this is divisible 
by the primes $r=7$ and $19$, both greater than $2n=6$. However, 
if $q=16\equiv 1$ mod~$(3)$ then $m=273=3\cdot7\cdot 13$, divisible by the prime 
$r=n=3$ in addition to $r=7$ and $13$. Note that the `large' primes $7, 13$ and $19$ 
appearing here as divisors of $m$ are all congruent to $1$ mod~$(2n)$.
\end{exam}


\subsection{Fixed $n$, while $q=p\to\infty$}\label{sec:fixed-n}
Let us fix a prime $n\ge 3$, and consider whether $m=(q^n-1)/(q-1)$ is prime.
For simplicity we will restrict $q$ to be prime, rather than a prime power; 
therefore, from now on we will denote it by $p$ instead of $q$. 
By the Prime Number Theorem (see~\cite[Theorem~6 and Ch.~XXII]{HW} for example),
the number of primes $p$ in the range 
$1\le p\le x$ is approximately $x/\ln(x)$ for large $x$. However, if $p\equiv 1$ mod~$(n)$ 
then $m$ cannot be prime by Lemma~\ref{SmallPrimes}, so we should restrict attention to the 
primes $p\not\equiv 1$ mod~$(n)$. Since primes are approximately evenly distributed 
between the non-zero congruence classes mod~$(n)$ (see~\cite[\S 5.3.2]{BS}, for example), 
the number of primes $p$ we should consider is therefore approximately $(n-2)x/(n-1)\ln(x)$.

Now $1\le m\le (x^n-1)/(x-1)$, and the probability that a randomly-chosen integer $m$
in this range is prime is approximately
\begin{equation}\label{prob1}
\frac{1}{\ln((x^n-1)/(x-1))}\approx\frac{1}{n\ln(x)-\ln(x)}=\frac{1}{(n-1)\ln(x)}.
\end{equation}
However, we know from Lemma~\ref{SmallPrimes} that $m\not\equiv 0$ mod~$(r)$ 
for each prime $r\le 2n$, including $r=n$ since $p\not\equiv 1$ mod~$(n)$.
For each such $r$, excluding this one congruence class mod~$(r)$ multiplies the 
probability of $m$ being prime by $r/(r-1)$. If we regard congruences modulo distinct 
primes as statistically independent, then we should multiply the probability 
in (\ref{prob1}) by $P(2n)$, where
\begin{equation}\label{prod}
P(y):=\prod_{{\rm prime}\,\,r\le y}\left(1-\frac{1}{r}\right)^{-1}
\end{equation}
for $y\ge 2$ and the product, as indicated, is over all primes $r\le y$. 
This gives an approximate probability
\begin{equation}\label{prob7}
\frac{P(2n)}{(n-1)\ln(x)}
\end{equation}
that $m$ is prime. For fixed $n$ this has the form $c_n/\ln(x)$ for a constant
\[c_n := \frac{P(2n)}{(n-1)}.\]
If $n$ is small one can easily calculate $c_n$: for instance $c_3=15/8$ and $c_5=35/32$. 
For large $n$ one can approximate $c_n$ by using a theorem of Mertens (see~\cite{Mer}, 
\cite[\S22.9]{HW} or~\cite[Theorem~8.8(e)]{NZM}) that
\[\prod_{{\rm prime}\,\,r\le y}\left(1-\frac{1}{r}\right)\sim
\frac{\mu}{\ln(y)}\quad{\rm as}\quad y\to\infty,\]
where $\mu:=e^{-\gamma}=0.561459\ldots$\ \label{const:mu}
and $\gamma$ is the Euler--Mascheroni constant $0.577215\ldots$.
This gives an approximate probability
\begin{equation}\label{prob3}
\frac{c_n}{\ln(x)}\sim\frac{e^{\gamma}\ln(2n)}{(n-1)\ln(x)}
\end{equation}
that $m$ is prime.

If we multiply this probability by the approximate number of primes $p\not\equiv 1$ mod~$(n)$ in the range $1\le p\le x$, namely $(n-2)x/(n-1)\ln(x)$, we see that the expected number of primes $m$ arising in this way is approximately
\begin{equation}\label{prob8}
\frac{c_n(n-2)x}{(n-1)(\ln(x))^2}\sim\frac{e^{\gamma}(n-2)\ln(2n)x}{(n-1)^2(\ln(x))^2}\approx\frac{1.781(n-2)\ln(2n)x}{(n-1)^2(\ln(x))^2}.
\end{equation}
Since this number tends to $+\infty$ as $x\to\infty$ for fixed $n$, this suggests 
that we should obtain infinitely many projective primes $m$ in this way for any fixed prime $n\ge 3$ 
(and likewise if we allow $q$ to be an arbitrary prime power). 

For each fixed $n$ the estimate in (\ref{prob8}) has the form $C_nx/(\ln(x))^2$ for some constant $C_n$ depending only on $n$. This is analogous to the Hardy--Littlewood estimate $Cx/(\ln(x))^2$ for the number $\pi_2(x)$ of twin prime pairs $p$, $p+2$ with $p\le x$ (see~\cite{PP}), where
\[C=2\cdot\negthinspace\negthinspace\prod_{{\rm prime\;} r\ge 3}\frac{r(r-2)}{(r-1)^2}\approx 1.320323632.\]


\begin{exam}[$n=3$]
If we take $n=3$, so that $c_n=15/8$, then the number of primes $m=1+p+p^2$ for 
primes $p\le x$ should be approximately
\begin{equation}\label{n=3estimate}
\frac{15x}{16 \ln(x)^2}
\end{equation}
for large $x$. This estimate is compared with computational evidence in
Section~\ref{primefields} (see Table~\ref{tab:n=3ratios}).
\end{exam}


\subsection{Fixed prime power $q=p^e$ with $e\ge 2$}\label{sec:fixed-q}

Instead, let us now fix $q$ and let $n\to\infty$.

\begin{lemm}\label{e>1}
If $e\ge 2$ then there are no projective primes $m=(q^n-1)/(q-1)$ with $n>e$.
\end{lemm}

\begin{proof}
If $e\ge 2$, so that $q$ is a prime power but not itself a prime, we have
$$
m = \frac{q^n-1}{q-1} = 
\frac{(1+p+\cdots+p^{n-1})\,(1+p^n+\cdots+p^{n(e-1)})}{1+p+\cdots+p^{e-1}}.
$$
This is clearly composite if $n>e$ since the two factors in the numerator are 
each larger than the denominator.
\end{proof}

Thus, for a fixed $q$ with $e\ge 2$ we can have only a finite number of projective primes $m$.

\begin{rema}[$e=2$]\label{re:e=2}
If $e=2$ and $m$ is prime then $n=2$ (and hence $q=p^2$ is even,
so that $p=2$ and $m=1+q=5$), against our earlier assumption;
thus in dealing with prime powers $q=p^e>p$ (as in Section~\ref{sec:primepowersagain},
where a number of examples are given) we will generally assume that $e\ge 3$.
\end{rema}

Since the ultimate goal of this line of research is to complete the classification 
of the permutation groups of prime degree, it would still be of interest to know the 
finitely many projective primes arising for each proper prime power $q$.


\subsection{Fixed prime $p$ while $n\to\infty$}\label{sec:fixed-p1}
Let us therefore take $e=1$, so we fix a prime $p$, and consider the primality of 
$m=(p^n-1)/(p-1)$ for odd primes $n$ as $n\to\infty$. By Lemma~\ref{SmallPrimes} 
we may exclude any primes $n$ dividing $p-1$, since they cannot give prime values 
of $m$. Then $m$ is not divisible by any prime $r\le 2n$.

By the Prime Number Theorem, for large $n$ a randomly-chosen integer close to $(p^n-1)/(p-1)$ is prime with probability approximately
\begin{equation}\label{prob}
\frac{1}{\ln((p^n-1)/(p-1))}\approx\frac{1}{n\ln(p)-\ln(p-1)}.
\end{equation}
However, $m$ is not uniformly distributed, since it is coprime to each prime 
$r\le 2n$. As before, for each such $r$ this excludes a proportion $1/r$ of the integers 
close to $(p^n-1)/(p-1)$, so we should multiply the probability in (\ref{prob}) by 
$M(2n)\sim e^{\gamma}\ln(2n)$, giving an approximate probability
\begin{equation}\label{prob2}
\frac{e^{\gamma}\ln(2n)}{n\ln(p)-\ln(p-1)}\sim\frac{e^{\gamma}\ln(n)}{n\ln(p)}
\end{equation}
that $m$ is prime, for large $n$. 
(For small $p$, as in Section~\ref{sec:fixed-p2}, this last approximation could induce 
significant errors.)

If we choose the prime $n$ uniformly and randomly from the range $p\le n\le x$ 
for some large $x$ (so that most such $n$ are large as above), then the expected number of 
primes $m$ arising is the sum of the probabilities in~(\ref{prob2}), that is
\[\frac{e^{\gamma}}{\ln(p)}\sum_n\frac{\ln(n)}{n}\]
where the sum is over all primes $n$ such that $p\le n\le x$. Now
\[\sum_n\frac{\ln(n)}{n}\approx\ln(x)-\ln(p-1),\]
(see~\cite[Theorem~425]{HW}), so the expected number of primes $m$ is approximately
\begin{equation}\label{qfixed}
\frac{e^{\gamma}(\ln(x)-\ln(p-1))}{\ln(p)}\sim\frac{e^{\gamma}\ln(x)}{\ln(p)}\approx\frac{1.781\ln(x)}{\ln(p)}.
\end{equation}
Since this tends to $+\infty$ with $x$, we may expect to obtain infinitely many 
projective primes $m$ from any given prime $p\ge 2$. This estimate is compared with 
computational evidence in Section~\ref{sec:fixed-p2} (see Table~\ref{tab:fixed-p}).


\subsection{A warning}\label{warning} 

Invoking the independence of congruences modulo different primes in order to make 
heuristic estimates, as we did in Section~\ref{sec:fixed-n}, has previously generated 
controversy: for instance, Wagstaff discusses this in~\cite{Wag}, citing criticism by 
Lenstra in~\cite{Len}. This is best illustrated with P\'olya's discussion in~\cite{Pol} 
of the following well-known paradox.

Based on the type of argument used in Section~\ref{sec:fixed-n}, one can attempt a heuristic proof of the Prime Number Theorem. An integer $x$ is prime if and only if $x\not\equiv 0$ mod~$(r)$ for each prime $r\le x$. For each such $r$ this event has probability $(r-1)/r$, so by regarding these events as mutually independent, and by using Mertens's Theorem,  one might expect $x$ to be prime with probability
\begin{equation}\label{PNT1}
\prod_{{\rm prime}\,\,r\le x}\left(1-\frac{1}{r}\right)\sim\frac{\mu}{\ln(x)}\quad\hbox{as}\quad x\to\infty,
\end{equation}
where $\mu=e^{-\gamma}=0.561459\ldots$. However, the correct asymptotic probability is $1/\ln(x)$, so this argument underestimates the probability of $x$ being prime (and hence the values of the prime-counting function $\pi(x)$) by a factor of $\mu$. Of course, it is sufficient to eliminate prime factors $r\le x^{1/2}$, rather than $r\le x$, so this alternative approach gives a second estimate
\begin{equation}\label{PNT2}
\prod_{{\rm prime}\,\,r\le x^{1/2}}\left(1-\frac{1}{r}\right)\sim\frac{\mu}{\ln(x^{1/2})}=\frac{2\mu}{\ln(x)}\quad\hbox{as}\quad x\to\infty.
\end{equation}
This overestimates the correct probability by a factor of $2\mu = 1.122918\ldots$, that is, by about $12\%$. If, as suggested by P\'olya in~\cite{Pol}, one takes the product over all primes $r\le x^{\mu}$ then the correct formula is obtained. P\'olya confesses that it is not clear why what he calls this ``trick of the magic $\mu$"  works here (Wagstaff~\cite{Wag} calls it a ``fudge factor"), but he goes on to argue that mathematicians should imitate physicists by adapting their theories to fit experimental data when such paradoxes arise. Similar phenomena are discussed by P\'olya~\cite{Pol} in relation to prime pairs and their generalisations, and by Wagstaff~\cite{Wag} in relation to the distribution of divisors of Mersenne numbers. 

The great Russian mathematician Andre\u{\i} Kolmogorov used to mention the following episode
(the second author heard it directly from him). Kolmogorov was once present at a talk
given by a prominent Russian physicist. The latter, basing his reasoning on some physical
ideas, introduced the density of a probability distribution on a certain space. Then,
he integrated this density and obtained $\pi$. At this point, Kolmogorov used to say, 
I would conclude that we had got a contradiction, and therefore all the reasoning was
wrong. But the conclusion of the physicist was different. Thus, he said, we must divide
the initial formula for the density by $\pi$. It seems that P\'olya would rather line up
with the physicist.

In our case, an appropriate choice of prime factors of $m$ to avoid is also an intricate 
matter. For example, for $n=3$, as we will see later, in Section~\ref{primefields}
and Table~\ref{tab:n=3ratios}, formula~(\ref{n=3estimate}) overestimates the number of 
projective primes. But we know that, 
beside the ``small primes'' 2, 3 and 5, Lemma~\ref{SmallPrimes}
also forbids all primes of the form $r \equiv -1$ mod~$(6)$. However, even if 
we adjoin to the product 
$$
P(6)=\left( \left(1-\frac{1}{2}\right)\left(1-\frac{1}{3}\right)
\left(1-\frac{1}{5}\right) \right)^{-1}
$$
not all such corresponding terms but only $(1-1/11)^{-1}$, we will get $c_3'=33/16$ 
instead of $c_3=15/8$, and estimate (\ref{n=3estimate}) will be replaced with 
$33x/32\ln(x)^2$, which will overestimate the number of projective primes even more 
than (\ref{n=3estimate}) does. On the other hand, if we remove the factor $(1-1/5)^{-1}$
from $P(6)$, we will underestimate the desired number.

But this is not yet the end of the story. As we will see in 
Section~\ref{sec:empiric} and Figure~\ref{fig:comparison},
the ratio of our estimates to the true number of projective primes grows, that is,
the estimates grow faster than the numbers they are supposed to estimate. 
There are even reasons to believe the overestimate grows to infinity 
(though slowly). Obviously, 
the constant factor does not play any part in this process: the rate of growth of the
estimates depends only on the behavior of the function $x/\ln(x)^2$. Therefore, it is 
reasonable to suppose that, for very large $x$, we will have indeed to eliminate 5 from 
the set of forbidden primes. Note however that this process of eliminating
or adjoining forbidden primes is not based on any solid theoretical foundation: it is
purely empirical. Therefore, instead of making artificial choices of which primes to 
include in the product, it seems at least as reasonable to consider other functions 
instead of  $x/\ln(x)^2$. This will be done in Section \ref{sec:empiric}.


\section{Primality testing}\label{sec:testing}

Before presenting the experimental results aiming to support our main conjecture
(that there are infinitely many projective primes), let us briefly discuss two problems:
the factorization of integers into prime factors, and the testing of primality. The problems
are, evidently, related to each other, but there is an abyss between their complexities.

\subsection{Integer factorization in modern times}

According to \cite{AGLL}, in 1977, Ronald Rivest, in a letter to 
Martin Gardner, estimated that
\begin{quote}
``{\ldots} factoring a 125-digit number
which is the product of two 63-digit 
prime numbers would require at least 40 quadrillion years using the best factoring algorithm 
known, assuming that $a\cdot b\; {\rm mod}\,(c)$ could be computed in 1~nano\-second, for 
125-digit numbers $a$, $b$, and $c$."
\end{quote}
It sounded like a solemn chorus from Purcell's {\sl Dido and Aeneas}\/: 
``Never! Never! Never!''
The same year, Gardner \cite{Gardner} launched a challenge: it was proposed to factor
a 129-digit number which was the product of a 64-digit and a 65-digit prime. Apparently
hopeless, whatever the future progress in computer technology would be.

Subsequent years saw spectacular progress in factorization 
algorithms. Finally, 17 years later, in 1994, the above 129-digit 
number was successfully factored. The project involved some 600 volunteers, 
1600 computers, and six months of computation. An account may be found 
in \cite{AGLL}; the strange expression ``squeamish ossifrage'' in the title was the 
message encrypted using this number by the RSA cryptographic method.

We skip a number of important developments during the next 15 years and go directly to a
milestone of 2009. In December of that year, a 232-digit number was factored: it was a
product two 116-digit primes. This result was the outcome of two years of work by a team
of 13 researchers, and was crowned with a {\$}50\,000 prize. As stated in~\cite{RSA},
\begin{quote}
``The CPU time spent on finding these factors by a collection of parallel computers 
amounted approximately to the equivalent of almost 2000 years of computing on a 
single-core 2.2 GHz AMD Opteron-based computer."
\end{quote}
Certainly, 2000 years for a 232-digit number as compared to 40\,000\,000\,000\,000\,000 
years for a 125-digit one, is incredible progress. A convenient, and machine-independent 
measure of an effort in a large-scale computation is GHz-years; in the present case we 
have 4400 GHz-years of computing.

Ten more years have passed, and a 240-digit number was factored in November 2019,
and a 250-digit number in February 2020. And here lies the current frontier of the
capability of factoring algorithms. The next challenge, a 260-digit number, still waits
for its turn to be factored.

\subsection{Testing}

The above examples show how difficult, in practical terms, the problem of factorization 
can be. However, there exist algorithms which establish whether a given integer is prime 
or composite, and this without ever trying to factor it. The most well-known, and the most 
used in practice, is the Rabin--Miller algorithm~\cite{Rabin} (see also the compendium 
\cite{CLR}). In particular, it is implemented in the Maple command {\em isprime}. Let us
take the above-mentioned 260-digit number (which, we recall, is not yet factored) and see
how this command works. The computation is carried out on a very modest laptop.

\begin{figure}[htbp]
\includegraphics[scale=0.9]{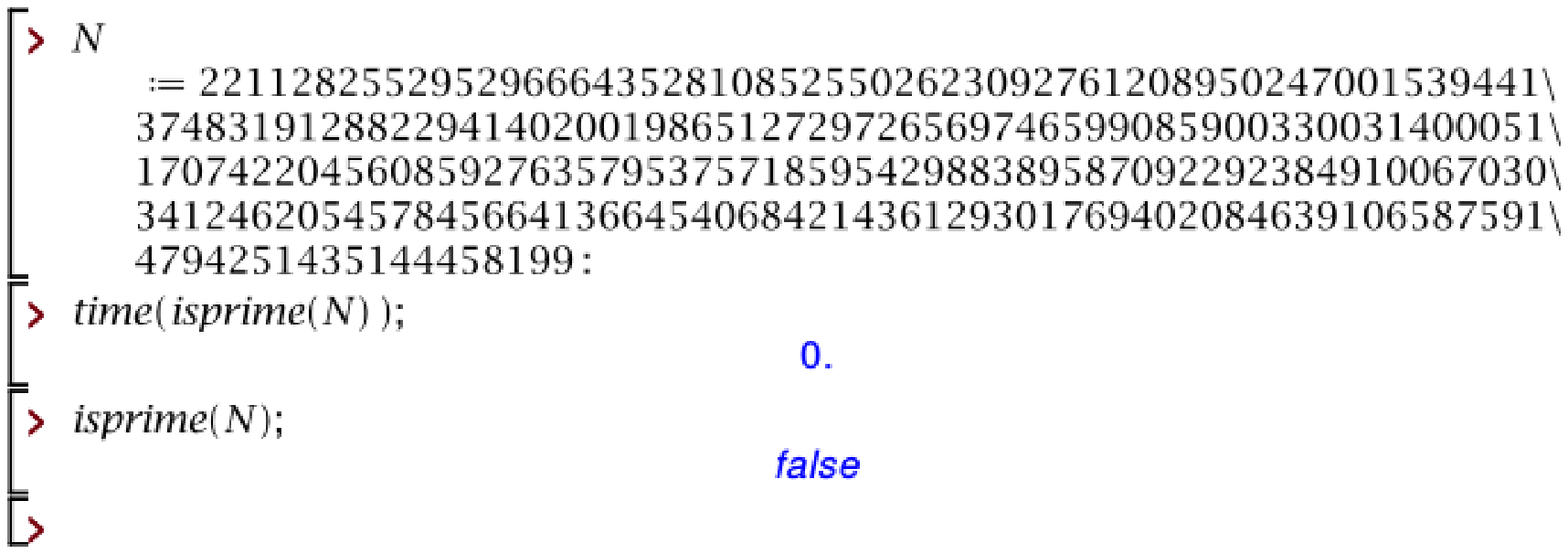}
\end{figure}

We see that the correct answer is given, almost literally, ``in no time at all''. 
In fact, the Maple time-counter outputs the CPU time within an accuracy of 0.001 seconds.
Therefore, 0. seconds time displayed in the above session means $<0.0005$ seconds, rounded
downwards.

Now consider a harder example, a 6153-digit number we will encounter in 
Section~\ref{1stseries}.

\begin{figure}[htbp]
\begin{flushleft}
\hspace*{4mm}
\includegraphics[scale=0.9]{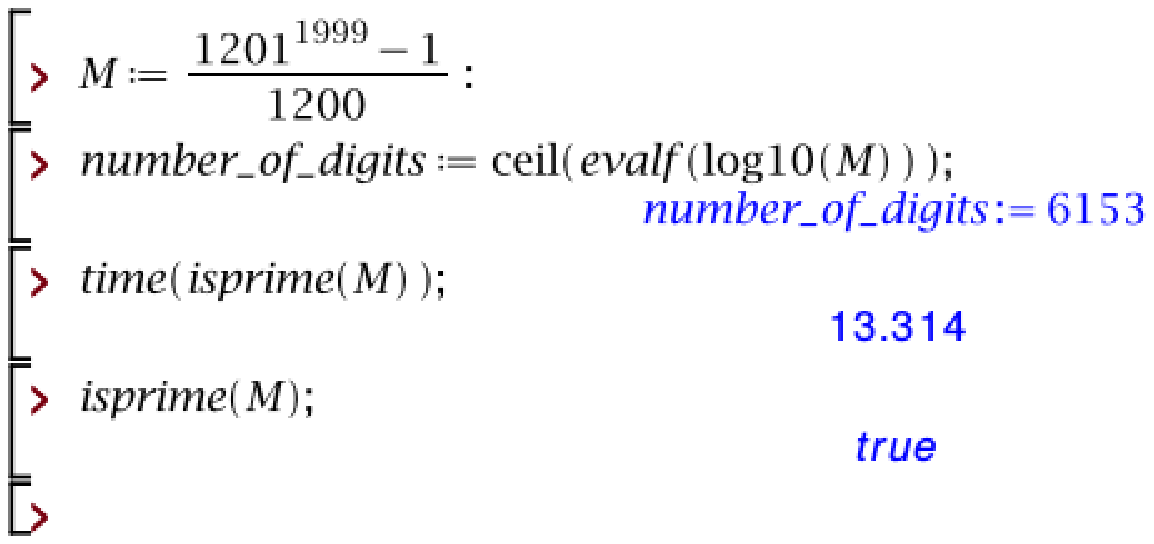}
\end{flushleft}
\end{figure}

This number is prime, and the computation took more than 13 seconds. A good result,
but to perform this testing on a large scale, that is, with large series of numbers, 
can turn out to be time-consuming.

\subsection{How the test works}

The Rabin--Miller algorithm is probabilistic. In order to determine whether a given 
number $m$ is prime it takes a random element $t\in\mathbb{Z}_m$ and verifies a necessary 
primality condition. The condition itself is simple, so we give it here.

The number $m-1$ is even; suppose it is equal to $m-1=(2l+1)\cdot 2^k$. Compute
in $\mathbb{Z}_m$
$$
a_0=t^{2l+1}, \quad \mbox{and then} \quad a_i=a_{i-1}^2 \quad \mbox{for} \quad
i=1,\ldots,k, \quad \mbox{so that} \quad a_k=t^{m-1}.
$$
If one of the following holds then $m$ is composite:
\begin{enumerate}
\item	While computing the sequence $a_i$, we come for the first time to $a_i=1$ 
		but the previous number $a_{i-1}\ne -1$. Indeed, in this case the equation 
		$a^2=1\;{\rm mod}\,(m)$ has, beside two obvious roots 1 and $-1$, a third 
		root $a_{i-1}$.
\item	We get $a_k=t^{m-1}\ne 1$. This contradicts Fermat's little theorem.
\end{enumerate}
Thus, if the test tells us that $m$ is composite then this statement is true, and no
probability is involved. If, however, neither of the two above conditions is satisfied,
we conclude that $m$ is {\em probably prime}. Rabin \cite{Rabin} showed that the
probability of an erroneous answer is bounded by~$1/4$; usually it is much smaller.
For large $m$, in the majority of cases this probability is infinitesimally small.
A dialogue from Gilbert and Sullivan's {\em I am the captain of the Pinafore}\/ comes
to mind: ``What, never? No, never. What, never? Well, hardly ever''. 
Nevertheless, in order to be on the safe side, the test is repeated many times with
different (random) values of~$t$. This, by the way, explains why the treatment of a
prime number takes much more time than that of a composite number of the same size.

Notice that raising a number $a$ to a power $a^r$ needs $O(\log r)$ arithmetic operations:
we compute first $a,a^2,a^4,a^8,\ldots$ (taking squares every time), and then multiply
the terms corresponding to the binary expression of the exponent $r$. Notice also that, 
in our case, all computations are made modulo $m$, so that the size of the numbers remains 
bounded.

\subsection{Polynomial-time algorithms}

The subject of primality testing and factorization has many ramifications. We only mention
very briefly a few of them. We recommend, for an interested reader, a very concise and
clear overview~\cite{Brent} and a more modern and advanced exposition in~\cite{FR}
(especially Chapter~5, ``Primality testing---an overview'').

There are several algorithms for primality testing whose complexity is polynomial in the 
size of tested numbers. However, for most of them the estimation of complexity is based
on some as yet unproved hypotheses.

\begin{nota}[Simplified measure of complexity]
Denote $k:=\log m$, and denote $\tO(k^s):=O(k^{s+\varepsilon})$ for all $\varepsilon>0$.
\end{nota}

This notation allows one to simplify complexity estimates for algorithms such as, for example,
the Sch\"onhage--Strassen algorithm of multiplication of long integers: we may now write 
just $\tO(k)$ instead of $O(k\cdot\log k\cdot \log\log k)$.

The complexity of the Rabin--Miller algorithm is $\tO(k^2)$: here $O(k)$ is the number
of arithmetic operations, and $\tO(k)$ is the complexity of an individual operation.

Four years before Rabin, Miller used the same test but in a deterministic way. Namely, 
it suffices to make the test for all $t\le 2\log(m)^2=2k^2$, {\em provided that the 
Extended Riemann Hypothesis is true}. Thus, this algorithm, of complexity is $\tO(n^4)$, 
while being deterministic, is based on an unproved conjecture. Also, the factor $2k^2$ 
is not innocuous. For $k\sim 6000$, as in the above example, it transforms seconds 
into years.

\begin{rema}[Are long computations reliable?]
It is important to note that in a long computation there is a significant probability 
of a hardware error. This probability is much greater than that in the Rabin--Miller test.
\end{rema}

In a revolutionary work \cite{Primes-in-P}, an {\em unconditional}\/ polynomial time 
algorithm for primality testing was given for the first time. Here `unconditional' means
that the estimate of its complexity does not depend on any unproved statement. After
several improvements its complexity is now established as $\tO(k^{15/2})$. It may also
be $\tO(k^6)$ if another as yet unproved conjecture is valid. Its theoretical impact 
is great but its practical utility is very limited.

Another method is based on the theory of elliptic curves. It is commonly known as the ECPP 
algorithm, which means Elliptic Curve Primality Proving. The names we must mention here 
are Sh.~Goldwasser, J.~Kilian, A.~Atkin and F.~Morain. This algorithm is probabilistic;
however, it is not of the ``Monte-Carlo type'' but of the ``Las Vegas type''. The latter means
that it always gives the correct answer; it is the computation time which is random. 
It is polynomial {\em on average}\/ if certain as yet unproved conjectures are true. 
Beside the correct answer, this algorithm also creates a {\em primality certificate}. 
A certificate is ``something'' which may be difficult to find but, once found, allows 
one to make a verification easily.

In \cite{Brent}, the following example is given. Consider the number 
$m=4405^{2638}+2638^{4405}$. It has 15\,071 digits. The {\em proof}\/ of its primality 
by the ECPP algorithm was achieved in 5.1~GHz-years. This is a truly remarkable 
result if we compare it with other error-free algorithms. Note, however, that the 
Rabin--Miller algorithm gives the correct (though unproved) answer in less than 
two minutes.

\subsection{A few comments}

Since 1980, when Michael Rabin published his algorithm, not a single case of an erroneous 
answer has been observed. Even financiers, in their cryptographic protocols, rely entirely on 
this test. However, a mathematical mind resists accepting a ``proof'' which in principle
might be wrong, even if the probability of such an event is infinitesimally small.
What then to do if we have doubts about the validity of the conclusion `prime' given 
by the probabilistic test? 
In our opinion, the most reasonable way to proceed is to run this test again once or 
twice. The test does not repeat exactly the same operations since it chooses different 
random elements of $\mathbb{Z}_m$ every time. In this way the probability $\alpha$ 
of an error, already infinitesimal, will be replaced with $\alpha^2$ or $\alpha^3$. 
(We may ask, rather provocatively: how many times can you repeat a two-minute test 
if you have 5.1 years at your disposal?)

And what if, by an incredible combination of chances, we take a composite number for 
a prime one? Well, let us recall that the aim of our particular study is to collect 
evidence that there are infinitely many projective primes. Therefore, one prime less 
or one prime more does not change much.

\section{Computational evidence}\label{Comp}

\subsection{Bunyakovsky's conjecture from an experimental perspective}
\label{sec:bun-experim}

The data in favor of the veracity of this conjecture abound. If we take, for example,
$f(t)=t^2+t+1$ and count the number of integers $t\le 10^7$ for which $f(t)$ is prime, 
we get 745\,582 solutions. A very ``modest'' particular case of Bunyakovsky's conjecture 
is known as Landau's conjecture: it concerns $f(t)=t^2+1$. In this case the number of
$t\le 10^7$ for which $t^2+1$ is prime is 456\,362. There is little doubt that,
at least in these two cases, the conjecture is true. No proof is, however, in view.

The main motivation of this note comes from group theory. Therefore, we will mainly 
consider not arbitrary values of $t$ but only prime powers $t=p^e$, $e\ge 1$,
and not arbitrary polynomials $f(t)$ but only those of the form
$$
f(t) = \frac{t^n-1}{t-1} = 1+t+t^2+\cdots+t^{n-1}.
$$

\begin{rema}[Terminological]
While speaking of {\em prime powers}, according to the context we may mean $p^e$ with
$e\ge 1$, that is, including ``pure'' primes, or, sometimes, with $e\ge 2$, in order
to put prime powers in contrast with the pure primes whose exponent is $e=1$.
\end{rema}

 
\subsection{First series of projective primes}\label{1stseries}

A computer search has revealed 
$668$ projective primes with $2\le q\le 2000$ and 
$3\le n\le 2000$, including one with $6153$ decimal digits, arising from $q=1201$ 
and $n=1999$. It is interesting to note that only five pairs $(q,n)$ out of $668$
correspond to prime powers $q=p^e$ with $e\ge 2$, namely,
$$
(q,n) \,=\, (2^3,3),\, (2^7,7),\, (2^9,3),\, (3^3,3),\, (11^3,3).
$$
All the other values of $q$ are ``pure'' primes.


\subsection{Number 31}\label{31}

A computer search of prime degrees up to $10^{12}$ reveals 
${\rm L}_3(5)$ and ${\rm L}_5(2)$ as the only pair of groups ${\rm L}_n(q)$ with 
the same natural degree in this range; it would be interesting to know whether any 
other such pairs exist. 

\begin{conj}[Number 31]\label{conj:31}
Beside\/ $31$, there are no other natural degrees common to two different projective 
groups ${\rm L}_n(q)$.
\end{conj} 

\begin{rema}[Goormaghtigh conjecture]
The Diophantine equation
$$
\frac{x^n-1}{x-1}=\frac{y^k-1}{y-1}
$$
has been studied by many authors (see~\cite[Problem B25]{Guy}, for example).
In 1917, a Belgian engineer and amateur mathematician Ren\'e Goormaghtigh\footnote{See 
{\tt https://forvo.com/word/ren\%C3\%A9\_goormaghtigh} to learn how to pronounce this name.} 
(1893--1960) conjectured \cite{Goo} that this equation, 
for $n\ne k$, $n,k\ge 3$, has only two solutions in $\mathbb N$:
$1+2+4+8+16=1+5+25=31$ and $1+2+4+\cdots+2^{12}=1+90+90^2=8191$.
However, 90 is not a prime power, so that there 
is no field with 90 elements. By the way, the number 8191 is prime. Therefore, it is
an instance of Bunyakovsky's conjecture for two different polynomials (and certainly 
for many other ones, like $t^2+91$, for example), but it is a projective prime for
only one of them, namely, $1+t+t^2+\cdots+t^{12}$. 

In \cite{DLS}, it is proved that for fixed exponents  $n\ne k$, $n,k\ge 3$, there can 
be only a finite number of solutions. For additional information about this equation 
see \cite{Goo-wiki}. 
\end{rema}


\subsection{Projective planes over prime fields}\label{primefields}

Let us take only prime values $p$, not taking into account the prime powers $q=p^e$ 
with $e\ge 2$, let us fix $n=3$ and consider projective primes $m=1+p+p^2$. Our colleague 
Jean B\'etr\'ema examined all primes $p\le 10^{11}$ using the package
{\tt Primes.jl} of the language {\tt Julia}. It turns out that {\tt Julia}
is much more efficient than Maple for problems of this sort. We partially reproduce 
B\'etr\'ema's results in Table \ref{tab:betrema}. 

\begin{table}[htbp]
\begin{center}
\begin{tabular}{l|c|c|c|c}
\hspace*{11mm}Segment & \#(prime $p$) &
\#(prime $m$) & ratio & $\max p$ \\
\hline
\hspace*{9mm} $2,\ldots,\, 10^{10}$ 			  & 455\,052\,511    &
15\,801\,827 & 3.473\% & 9\,999\,999\,491  \\
\hspace*{4mm} $10^{10},\ldots,\, 2\cdot 10^{10}$  & 427\,154\,205    & 
13\,882\,936 & 3.250\% & 19\,999\,999\,757 \\
$2\cdot 10^{10},\ldots,\, 3\cdot 10^{10}$ 		  & 417\,799\,210    & 
13\,279\,095 & 3.178\% & 29\,999\,999\,921 \\
$3\cdot 10^{10},\ldots,\, 4\cdot 10^{10}$ 		  & 411\,949\,507    & 
12\,913\,713 & 3.135\% & 39\,999\,999\,719 \\
$4\cdot 10^{10},\ldots,\, 5\cdot 10^{10}$ 		  & 407\,699\,145    & 
12\,645\,233 & 3.102\% & 49\,999\,999\,619 \\
$5\cdot 10^{10},\ldots,\, 6\cdot 10^{10}$ 		  & 404\,383\,577    & 
12\,439\,618 & 3.076\% & 59\,999\,999\,429 \\
$6\cdot 10^{10},\ldots,\, 7\cdot 10^{10}$		  &	401\,661\,384    &
12\,274\,191 & 3.056\% & 69\,999\,999\,287 \\
$7\cdot 10^{10},\ldots,\, 8\cdot 10^{10}$		  &	399\,359\,707 	 &
12\,136\,112 & 3.039\% & 79\,999\,999\,679 \\
$8\cdot 10^{10},\ldots,\, 9\cdot 10^{10}$		  & 397\,369\,745	 &
12\,010\,780 & 3.023\% & 89\,999\,999\,981 \\
$9\cdot 10^{10},\ldots,\, 10^{11}$  			  & 395\,625\,822	 &
11\,910\,803 & 3.011\% & 99\,999\,999\,977 \\
\hline
\hspace*{13mm}Total 	  						  & 4\,118\,054\,813 &
129\,294\,308 & 3.140\% & 99\,999\,999\,977
\end{tabular}
\end{center}
\vspace{2mm}
\caption{The second column gives the number of primes in the corresponding
segment, while the third column gives the number of those primes $p$ which
create a projective prime $m=1+p+p^2$. The proportion of such primes among all the
primes of the second column is given in the fourth column.}
\label{tab:betrema}
\end{table}

We may see from this table that the number of primes $p\le 10^{11}$ which produce a 
prime value of $m$ is 129\,294\,308, the largest of them being 99\,999\,999\,977. 
The corresponding projective prime is $m=9\,999\,999\,995\,500\,000\,000\,507$.
Such primes $p$ represent approximately 3.140\% of the total number 
4\,118\,054\,813 of primes up to $10^{11}$.

Of course, this percentage diminishes together with the growth of the upper limit. 
For example, if we count the proportion of such primes up to $10^6$, we get
5.97\%. Nevertheless, it is quite reasonable to conjecture that even in this very 
restricted situation there are infinitely many projective primes.

\begin{table}[htbp]
\begin{center}
\begin{tabular}{c|c|c|c|c|c}
$x$ & \#(prime $m \mid p\le x$) & estimate (\ref{n=3estimate}) & ratio 
& estimate (\ref{eq:rectify}-\ref{eq:C-and-alpha}) & ratio \\
\hline
$1\cdot 10^{10}$ & 15\,801\,827 & $1.7683\times 10^7$  & 1.1190 
& $1.5799306\times 10^7$ & 0.999841 \\
$2\cdot 10^{10}$ & 29\,684\,763 & $3.3328\times 10^7$  & 1.1227 
& $2.9686686\times 10^7$ & 1.000065 \\
$3\cdot 10^{10}$ & 42\,963\,858 & $4.8096\times 10^7$  & 1.1195 
& $4.2969637\times 10^7$ & 1.000135 \\
$4\cdot 10^{10}$ & 55\,877\,571 & $6.2736\times 10^7$  & 1.1227 
& $5.5881270\times 10^7$ & 1.000066 \\
$5\cdot 10^{10}$ & 68\,522\,804 & $7.7239\times 10^7$  & 1.1272 
& $6.8526763\times 10^7$ & 1.000058 \\
$6\cdot 10^{10}$ & 80\,962\,422 & $9.1332\times 10^7$  & 1.1281 
& $8.0965961\times 10^7$ & 1.000044 \\
$7\cdot 10^{10}$ & 93\,236\,613  & $1.0524\times 10^8$ & 1.1287 
& $9.3237376\times 10^7$ & 1.000008 \\
$8\cdot 10^{10}$ & 105\,372\,725 & $1.1900\times 10^8$ & 1.1293 
& $1.0536780\times 10^8$ & 0.999953 \\
$9\cdot 10^{10}$ & 117\,383\,505 & $1.3262\times 10^8$ & 1.1298 
& $1.1737691\times 10^8$ & 0.999944 \\
$10^{11}$ & 129\,294\,308        & $1.4614\times 10^8$ & 1.1303 
& $1.2927974\times 10^8$ & 0.999887
\end{tabular}
\end{center}
\vspace{2mm}
\caption{The second column gives the cumulative totals from the second column in 
Table~\ref{tab:betrema}, i.e.~the number of projective primes $m$ with $n=3$ arising 
from primes $p\le x_i=i\cdot 10^{10}\;(i=1,\ldots, 10)$; the third column gives an 
approximation for the estimate for this number from Section~\ref{sec:fixed-n}, 
while the fourth column gives the ratio of these two numbers. The meaning of the
last two columns is explained in Section~\ref{sec:empiric}.}
\label{tab:n=3ratios}
\end{table}

Table~\ref{tab:n=3ratios} compares the numbers of projective primes $m=1+p+p^2$ for 
primes $p\le x_i=i\cdot 10^{10}$, $i=1,\ldots,10$, with the heuristic estimates given 
by (\ref{n=3estimate}) in Section~\ref{sec:fixed-n}. It can be seen that the latter 
are of the right order of magnitude, but that they consistently over-estimate the 
number of such primes by about $12\%$. 
In Section~\ref{sec:empiric} we present
another estimate. We do not have any theoretical bases to support it, only empirical
ones, but it approximates the values we need much better than the previous estimate.
The results of this estimate are represented in the two last columns of 
Table~\ref{tab:n=3ratios}.


\subsection{An empirical estimate vs.~the theoretical one: rectifying the anomaly}
\label{sec:empiric}

We see that for large $x$ estimate (\ref{n=3estimate}) systematically overestimates the 
number of projective primes. Hence, let us instead consider an estimate of the following 
form:
\begin{eqnarray}\label{eq:rectify}
y = \frac{C x}{\ln(x)^{\alpha}}
\end{eqnarray}
where the constants $C$ and $\alpha$ are to be found from empirical data.

Denote $z=x/y=\frac{1}{C}\ln(x)^{\alpha}$. Then, taking the logarithm of each side of
this equation we get 
$$
\ln(z)=-\ln(C)+\alpha\cdot\ln(\ln(x)).
$$
Thus, in the coordinates 
$$
u=\ln(\ln(x)), \qquad  v=\ln(z) 
$$ 
equation (\ref{eq:rectify}) takes the form of an equation of a straight line
$$
v = a+\alpha u \quad\mbox{ where }\quad a=-\ln(C).
$$

Our next steps are as follows:
\begin{enumerate}
\item	Take a number of pairs $(x_i,y_i)$ where $x_i$ are at our choice while $y_i$ 
		are the numbers of projective primes $m=1+p+p^2$ created from primes $p\le x_i$.
\item	Compute the corresponding pairs $(u_i,v_i)$.
\item	Put the points $(u_i,v_i)$ on the plane of the coordinates $(u,v)$, in the hope
		that they will be reasonably close to a straight line $v=a+\alpha u$.
\item	Find the equation of this straight line and thus obtain the values 
		of $C=e^{-a}$ and of $\alpha$.
\end{enumerate}

The points should be taken with care. As we will see in Section \ref{sec:stochastic},
numbers of projective primes behave rather randomly. Therefore, in order to get a 
reasonable estimate, the values of $x_i$ must be large enough, and the spaces between 
them must also be large. 

The numbers $x_i=i\cdot 10^{10}$, $i=1,\ldots,10$ satisfy both conditions. Taking $x_i$ 
and $y_i$ from the first two columns of Table~\ref{tab:n=3ratios} we get the results 
which are shown in Figure~\ref{fig:empiric-alpha}. We would say that they are even better 
than one might hope. 
The corresponding constants, found by the method of least squares, are as follows:
\begin{eqnarray}\label{eq:C-and-alpha}
\hspace*{10mm}
a = -0.150383694, \qquad C = e^{-a} = 1.162280117, \qquad \alpha = 2.104419156.
\end{eqnarray}
The estimates given by (\ref{eq:rectify}) with the constants (\ref{eq:C-and-alpha}), 
and their ratios to the true number of projective primes are shown in the last two 
columns of Table~\ref{tab:n=3ratios}.

\begin{rema}[Overestimate]
If the estimates (\ref{eq:rectify}-\ref{eq:C-and-alpha}) are correct, and it seems 
that they are, or at least if they are close to the correct asymptotic, then the 
{\em overestimate}\/ of the ratio given by formula (\ref{n=3estimate}) tends to infinity, 
though rather slowly: it is proportional to $\ln(x)^{\alpha-2}$.
\end{rema}

\begin{figure}[htbp]
\begin{center}
\includegraphics[scale=0.3]{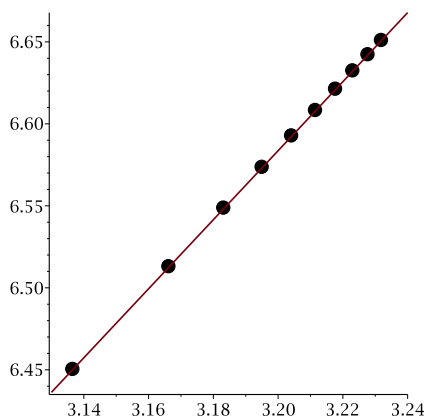}
\caption{The horizontal axis corresponds to the variable $u=\ln(\ln(x))$, the vertical
one to the variable $v=\ln(x/y)$ where $y$ is meant to count projective primes. The ten
distinguished points correspond to $x_i=i\cdot 10^{10}$, while $y_i$ is the number of
prime $p\le x_i$ such that $m=1+p+p^2$ is prime.}
\label{fig:empiric-alpha}
\end{center}
\end{figure}

\begin{figure}[htbp]
\begin{center}
\includegraphics[scale=0.3]{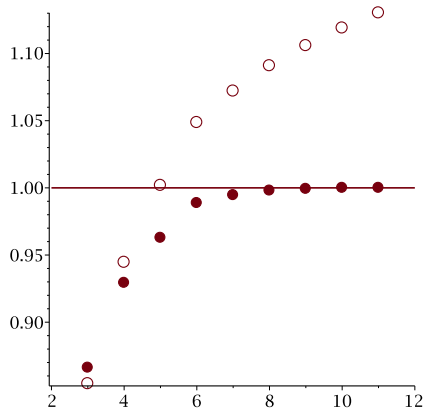}
\caption{Comparison of two estimates. The horizontal scale is logarithmic: an abscissa $k$
corresponds to the number $10^k$. White circles represent the ratios of estimates
(\ref{n=3estimate}) to the true numbers of projective primes $m=1+p+p^2$ obtained from 
the primes $p\le 10^k$, $k=3,\ldots,11$. The solid circles represent the similar ratios for 
estimates (\ref{eq:rectify}-\ref{eq:C-and-alpha}). Notice that for smaller numbers both
estimates {\em underestimate}\/ the number of projective primes (both black and white 
points are below the level 1).}
\label{fig:comparison}
\end{center}
\end{figure}

It would be interesting to understand the nature of the above constants. For example,
$\alpha$ is reasonably close to $1+2\mu=2.1229$ (the constant $\mu$ is defined in 
Section~\ref{sec:fixed-n}). It would be tempting to conjecture that they are equal, especially
since $\mu$ appears in several other conjectures related to prime numbers (see, for 
example, \cite{Pol}, and also the chapter on the Mersenne primes in \cite{PP}). But no:
experience shows that the estimate of the exponent $\alpha$ diminishes when the
the bound $x$ grows. And let us not forget that for the time being we are unable to 
prove even that there are infinitely many projective primes, to say nothing of their 
asymptotic behavior.


\subsection{When the exponent $n$ grows}

We considered the prime exponents $n\le 100$ and counted the number of primes $p\le 10^6$
(omitting prime powers) such that $m=(p^n-1)/(p-1)$ is also prime. The results are
presented in Table~\ref{tab:growing-n}: $N$ denotes the number of primes $p$ with the
above property, and $\max p$ is the largest $p\le 10^6$ which, for a given $n$, produces
a prime value of $m$. The total number of primes up to one million is 78\,498. The proportion
of ``good'' primes thus varies (in our table) from 6\% (for $n=3$) to 0.5\% (for $n=83$).
There is a general tendency for this proportion to decrease, but without any apparent
regularity.

\begin{table}[htbp]
\begin{center}
\begin{tabular}{c|c|c||c|c|c||c|c|c||c|c|c}
$n$ & $N$  & $\max p$ & 
$n$ & $N$  & $\max p$ & 
$n$ & $N$  & $\max p$ & 
$n$ & $N$  & $\max p$ \\
\hline
 3  & 4684 & 999\,773 & 
19  & 2933 & 999\,067 & 
43  & 1119 & 999\,961 & 
71  &  848 & 999\,907 \\
 5  & 4034 & 999\,653 & 
23  & 1150 & 999\,287 & 
47  & 1212 & 999\,491 & 
73  &  577 & 999\,307 \\
 7  & 4436 & 999\,961 & 
29  & 1032 & 998\,111 & 
53  &  694 & 999\,007 & 
79  &  689 & 996\,811 \\
11  & 2243 & 999\,631 & 
31  & 1980 & 997\,463 & 
59  & 1106 & 999\,953 & 
83  &  390 & 993\,557 \\
13  & 2658 & 999\,863 & 
37  & 1285 & 999\,269 & 
61  &  913 & 999\,763 & 
89  &  430 & 995\,339 \\
17  & 2527 & 999\,287 & 
41  &  862 & 999\,233 & 
67  &  821 & 999\,727 & 
97  &  571 & 998\,471 
\end{tabular}
\end{center}
\vspace{2mm}
\caption{For a given prime exponent $n\le 100$, the number $N$ shows how many primes 
$p\le 10^6$ there are such that $m=(p^n-1)/(p-1)$ is also prime. The column ``$\max p$'' 
shows the largest such $p$.}\label{tab:growing-n}
\end{table}

We do not present the corresponding projective primes $m$ since their decimal 
representations are too long: for example, the number $m$ corresponding to the last 
cell of the table, namely, \linebreak
$m=(998\,471^{97}-1)/998\,470$, has 576 digits.

The following conjecture seems quite reasonable:

\begin{conj}[Projective primes for a fixed $n$]
For any fixed prime $n\ge 3$ there are infinitely many prime values $m=1+p+p^2+\ldots+p^{n-1}$,
where $p$ ranges over all prime numbers.
\end{conj}


\subsection{Projective primes with a fixed $p$}\label{sec:fixed-p2}

What if we fix $p$ and allow $n$ to tend to infinity (taking only prime values),
as in Section~\ref{sec:fixed-p1}? Here the computational evidence, presented in Table~\ref{tab:fixed-p}, is less convincing, which is not surprising given the small number of primes $m$ involved.

\begin{table}[htbp]
\begin{center}
\begin{tabular}{c|c|c|l}
$p$ & estimation (\ref{qfixed}) & true number & exponents $n$ \\
\hline
3  & 13.832 & 12 & $3,7,13,71,103,541,1091,1367,1627,4177,9011,9551$ \\
5  & 8.669  & 11 & $3,7,11,13,47,127,149,181,619,929,3407$ \\
7  & 6.796  & 5  & $5,13,131,149,1699$  \\
11 & 5.136  & 9  & $17,19,73,139,907,1907,2029,4801,5153$ \\
13 & 4.675  & 9  & $5,7,137,283,883,991,1021,1193,3671$ \\
17 & 4.052  & 7  & $5,7,11,47,71,419,4799$ \\
\hline
\end{tabular}
\vspace{2mm}
\caption{The last column gives the list of exponents $n\le 10^4$ such that the number
$m=(p^n-1)/(p-1)$ is prime. The number of such exponents is given in column~3,
while the estimation of this number by formula (\ref{qfixed}) is presented in column~2.}
\label{tab:fixed-p}
\end{center}
\end{table}

For the primes $p=3,5,7,11, 13$ and $17$, the second column of Table~\ref{tab:fixed-p} 
gives the estimates based on the first expression in (\ref{qfixed}) for the number of 
primes $m=(p^n-1)/(p-1)$ with $n\le x=10^4$. The third column gives the true figures, 
found by a computer search, and the relevant exponents $n$ are listed in the fourth column.

We see that the estimation (\ref{qfixed}) of the number of exponents has a reasonably 
good correspondence with their actual number. However, we have a feeling that the above
data are not entirely convincing. For example, there are only three exponents in the
table, out of 53, which are greater than 5000. Also, for $p=3$, the next ``good'' exponent, 
after the one given in the table, is rather far away: $n=36\,913$. Therefore, in this 
case we prefer to formulate not a conjecture but a question:

\begin{ques}[Generalized Mersenne]
Let $p$ be a prime. Do there exist infinitely many values of~$n$ such that the number
$m=(p^n-1)/(p-1)$ is prime?
\end{ques}


\subsection{Prime powers once again}\label{sec:primepowersagain} 

We now consider fixed prime powers $q=p^e$ with $e\ge 2$, as $n\to\infty$. By 
Remark~\ref{re:e=2}, apart from the example $q=2^2$ and $m=5$ we can ignore the 
case $e=2$. An extensive search of prime powers producing projective primes has 
given the following results: 

\medskip

\noindent
$p=2$: the search for $q=2^e$, $e\ge 2$, producing projective primes,
		up to $q\le 10^{60}$, gives eight solutions: 
		\begin{itemize}
		\item	There are four solutions for which $m=1+q$\, is a Fermat prime: 
				$$
				(q,n)=(2^2,2),\,(2^4,2),\,(2^8,2),\,(2^{16},2)
				$$ 
				(the other known Fermat prime $m=1+2^1=3$ does not correspond to 
				$e\ge 2$).
		\item	There are three relatively small solutions: $(q,n)=(2^3,3),(2^7,7),(2^9,3)$.
		\item	A rather unexpected solution is $(q,n)=(2^{59},59)$. The corresponding
				projective prime
				$$
				m = 1 + 2^{59} + 2^{118} + \cdots + 2^{59\cdot 58}
				$$
				has 1031 digits.
				It is generally believed that there are only five Fermat primes.
				However, this example prevents us from conjecturing that there are
				only finitely many powers of~2 which yield projective primes.
		\end{itemize}

\medskip

\noindent
$p=3$: the search for $q=3^e$, $e\ge 3$, producing projective primes,
		up to $q\le 10^{60}$, gives only one solution: $(q,n)=(3^3,3)$, $m=1+27+27^2=757$. 
		We are inclined to believe that for $p=3$ this solution is unique.
		
\medskip

\noindent
$q\le 10^{15}$: the total search for all prime powers $q\le 10^{15}$ 
		with $e\ge 3$ producing projective primes gives 337 solutions. 
		Only eight of them have the exponent $e>3$, namely,
		$$
		(q,n)=(5^7,7),\,(11^9,3),\,(43^5,5),\,(67^7,7),\,(167^5,5),\,(313^5,5),\,(509^5,5),\,
		(859^5,5).
		$$
		For all the other 329 solutions $q$ is the cube of a prime.
		
\medskip

\noindent
$q=p^3\le 10^{18}$: the total search for cubes of primes up to
		$10^{18}$ reveals 2121 solutions, the largest one being $p=999\,953$, 
		$q = p^3 = 999\,859\,006\,626\,896\,177$, and
		$$
		m = 1+q+q^2 = 999\,718\,033\,132\,923\,614\,193\,697\,947\,364\,111\,507.
		$$

\medskip

The following conjecture seems to be very plausible:

\begin{conj}[Cubes of primes]
There are infinitely many values of $q=p^3$, with $p$ being prime, such that
$m=1+q+q^2$ is prime.
\end{conj}

Since we did not find any examples where the same prime power $q$ yields more than one 
projective prime, we ask:

\begin{ques}[Generalized Mersenne]
Does there exist a proper prime power $q$ such that the number
$m=(q^n-1)/(q-1)$ is prime for more than one value of $n$?
\end{ques}

Of course, Table~\ref{tab:fixed-p} in Section~\ref{sec:fixed-p2} gives a number of examples of this phenomenon where $q=p$ is prime.


\section{Stochastic behavior of projective primes}
\label{sec:stochastic}

In this section, we present a number of observations concerning the stochastic
behavior of the (numbers of) projective primes. Our feeling is that this subject,
while being excitingly interesting, is not yet ready for a profound statistical 
analysis. However, we would like to share our observations with the community of 
specialists in probabilistic number theory in the hope that they may clarify 
certain points of our study.

We deal here exclusively with projective primes of the form $m=1+p+p^2$, where $p$ is prime.


\subsection{Local estimates}
\label{sec:segments}

Let $a,b$ be two integers, $a<b$. Then, according to (\ref{eq:rectify}-\ref{eq:C-and-alpha}), 
the primes $p\in[a,b]$ should create, approximately,
\begin{equation}\label{eq:est-by-segm}
C\cdot\left(\frac{b}{\ln(b)^{\alpha}}-\frac{a}{\ln(a)^{\alpha}}\right)
\end{equation}
projective primes of the type $m=1+p+p^2$, where $C = 1.162280117$ and 
$\alpha = 2.104419156$. In Figure~\ref{fig:est-by-segm}, left, we 
consider the primes $p\le 10^9$. We subdivide this range into $10^4$ segments
$S_i=[(i-1)\cdot 10^5, i\cdot 10^5]\;(i=1, 2, \ldots, 10^4)$ of equal size $10^5$. 
Horizontally, we mark the order number $i$ of a segment (from 1 to $10^4$). 
For each of these segments, we divide estimate (\ref{eq:est-by-segm}) by the true 
number of projective primes $m=1+p+p^2$ for $p\in S_i$; this ratio is the ordinate 
of the corresponding point in the picture. The picture thus contains $10^4$ points.

On the right of Figure \ref{fig:est-by-segm}, the construction is similar, but now the
range considered is $p\le 10^{10}$, and the length of each of the $10^4$ segments
is $10^6$.

\begin{figure}[htbp]
\begin{center}
\includegraphics[scale=0.35]{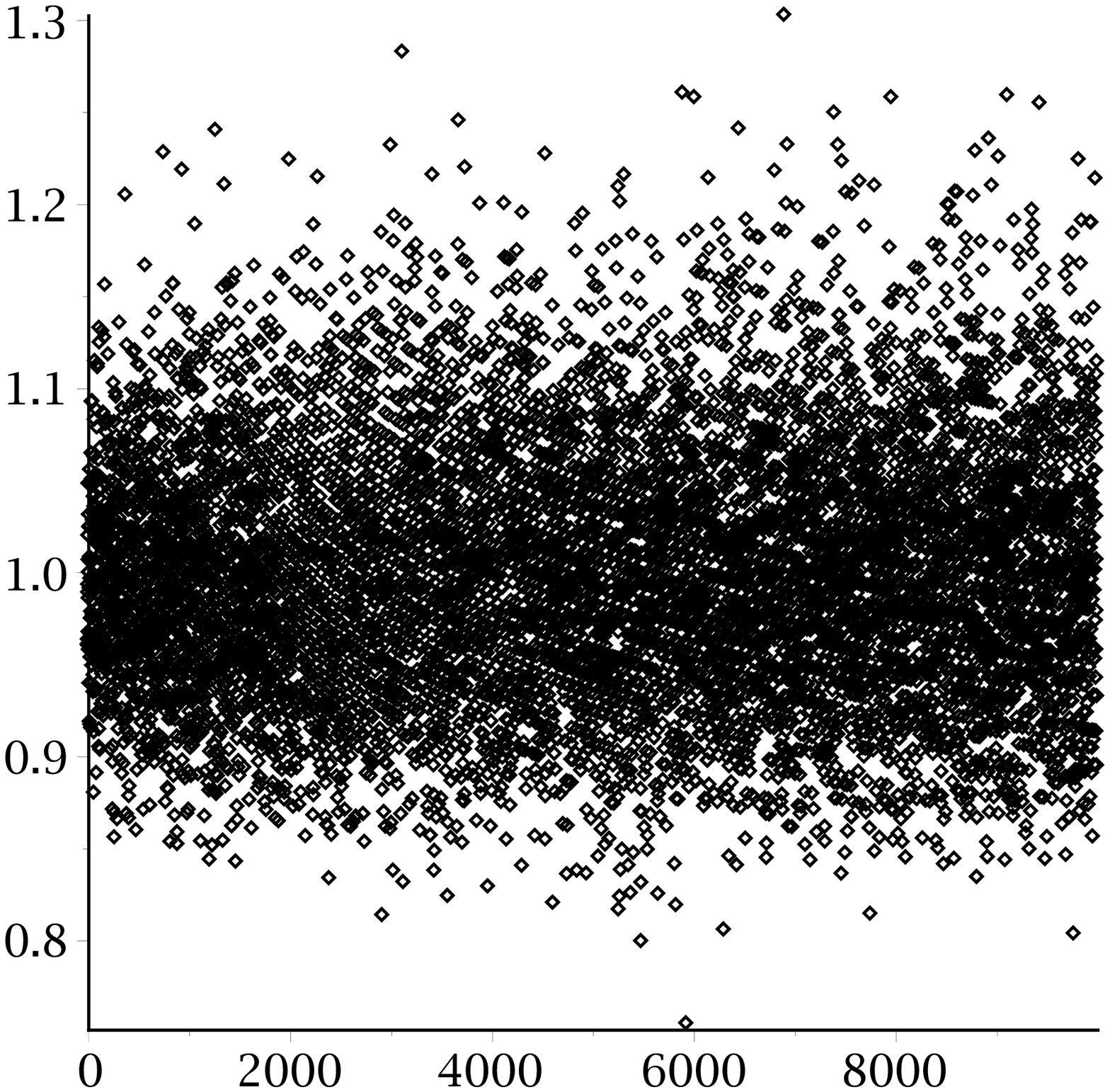}
\hspace{1cm}
\includegraphics[scale=0.35]{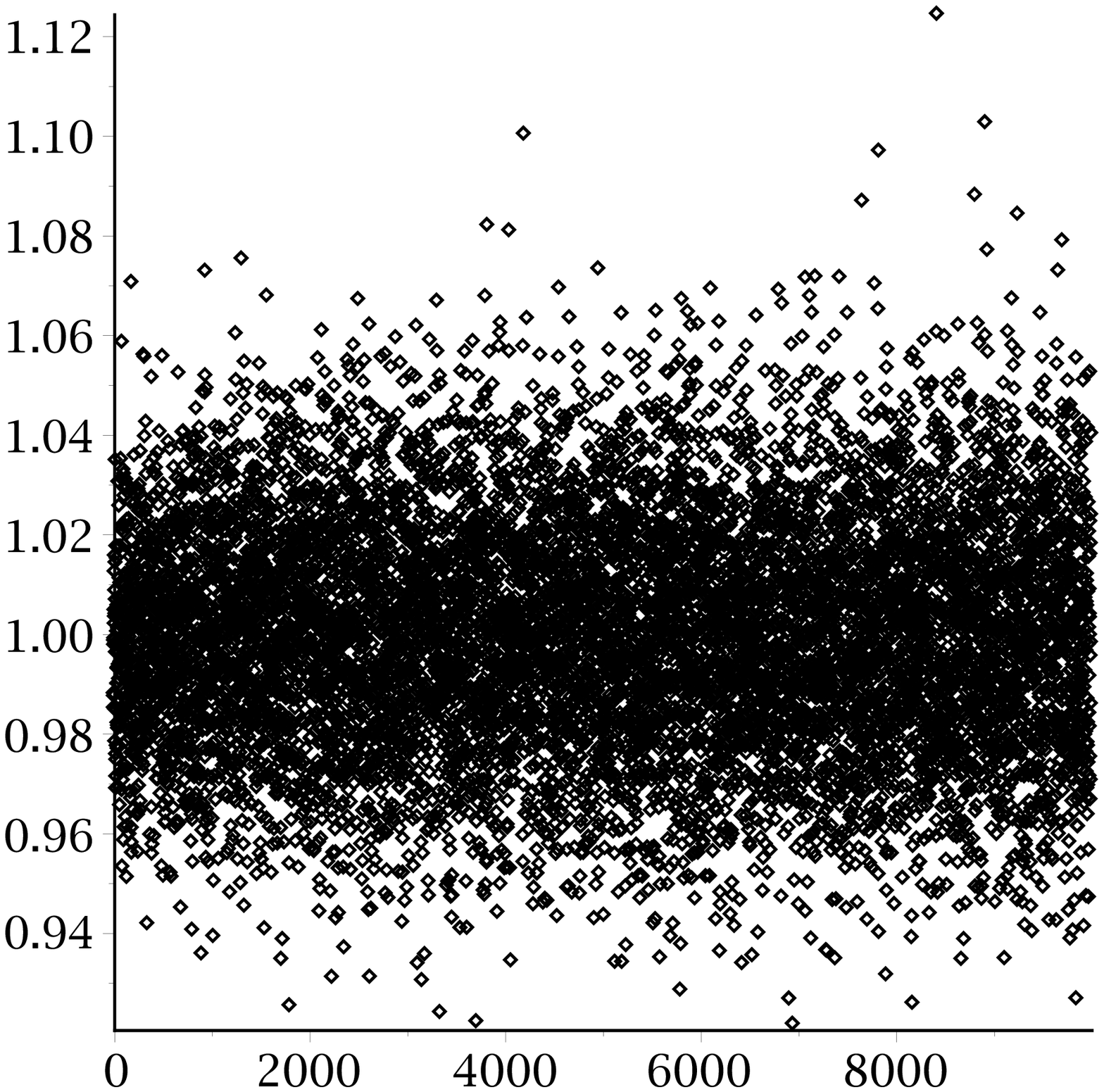}
\caption{Both pictures contain $10^4$ points. Each point corresponds to a segment 
in $\mathbb{N}$ of length $10^5$ (left) or $10^6$ (right). The abscissa of a point
is the order number of the corresponding segment. The ordinate is the ratio of the
estimate (\ref{eq:est-by-segm}) to the true number of projective primes
$m=1+p+p^2$ generated by primes $p$ in this segment.}
\label{fig:est-by-segm}
\end{center}
\end{figure}

We may make the following observations.\ 
\begin{itemize}
\item	There are wild fluctuations in the ratio. Obviously, they are due not to 
		estimate (\ref{eq:est-by-segm}) itself but to the fluctuations in the numbers 
		of projective primes $m=1+p+p^2$ with $p$ belonging to the corresponding 
		segments.
\item	Comparing the vertical scales shows that the right-hand band of points is 
		narrower than the left-hand one. This is
		natural since considering larger segments leads to smoothing the fluctuations.
\item	The interesting fact is, however, that the variations in both pictures do not
		diminish when we let $i$ increase. We may even say that they increase.
\end{itemize}


\subsection{Histogram}\label{sec:histo}

Let us take the right-hand picture of Figure~\ref{fig:est-by-segm}. The minimum value
of the ordinate (i.\,e., of the ratio) in this picture is $r_{\rm min}=0.9217$, the
maximum is $r_{\rm max}=1.1244$. We subdivide the segment $[r_{\rm min},r_{\rm max}]$
into 100 parts and count the number of points whose ordinates belong to each part.
The resulting histogram is shown in Figure~\ref{fig:density}, left.

The mean of this distribution is $46.66$, the standard deviation is $13.66$.
On the right of the same figure we show the density of the normal distribution with
the same parameters. Note that the height of the left picture, which is 300 points out
of 10\,000, corresponds well to the height of the density, which is approximately 0.03.

The resemblance of the two graphs is visible. We leave it to the specialists to use,
if necessary, more sophisticated statistical tools.

\begin{figure}[htbp]
\begin{center}
\includegraphics[scale=0.3]{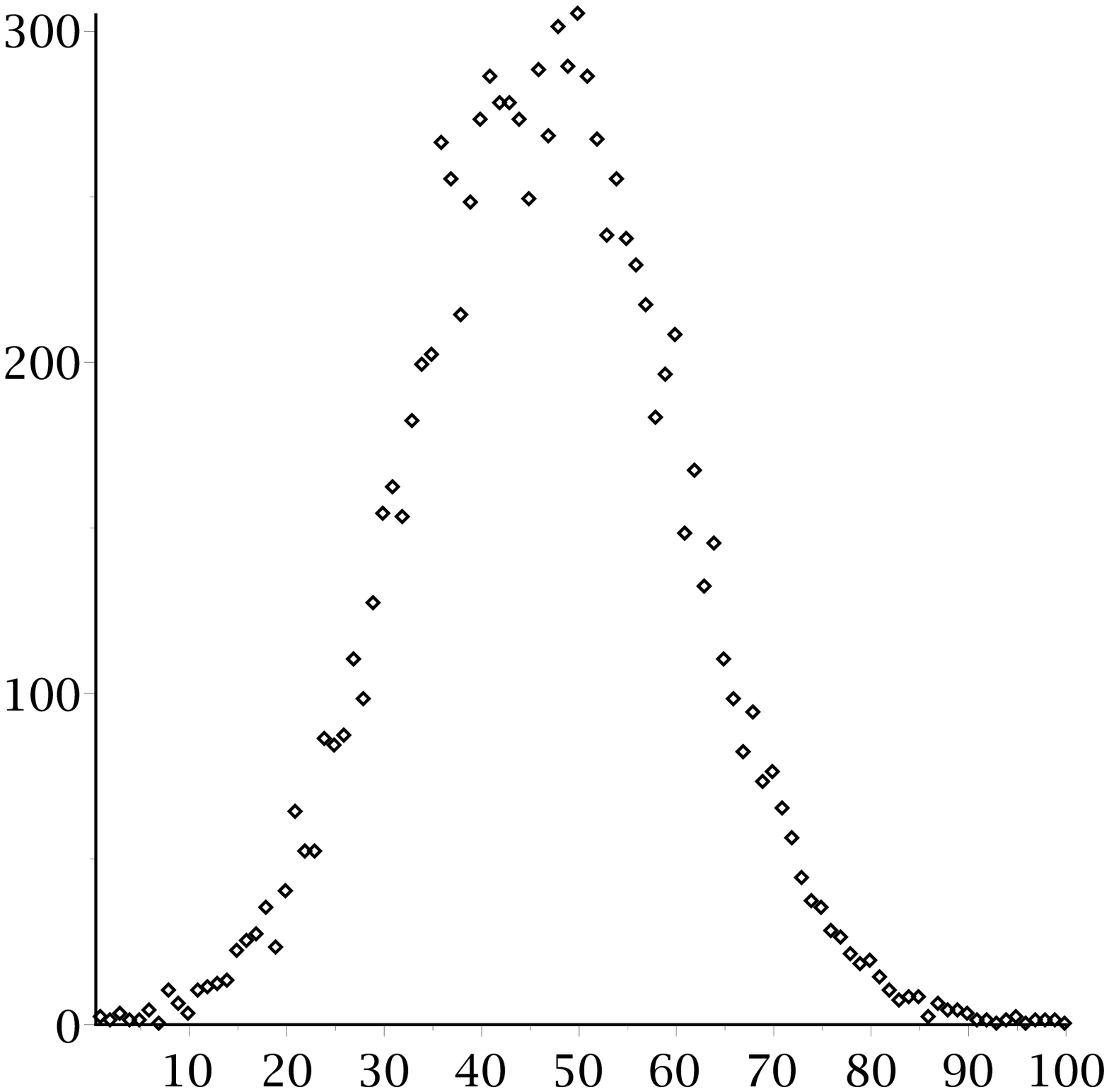}
\hspace{1cm}
\includegraphics[scale=0.3]{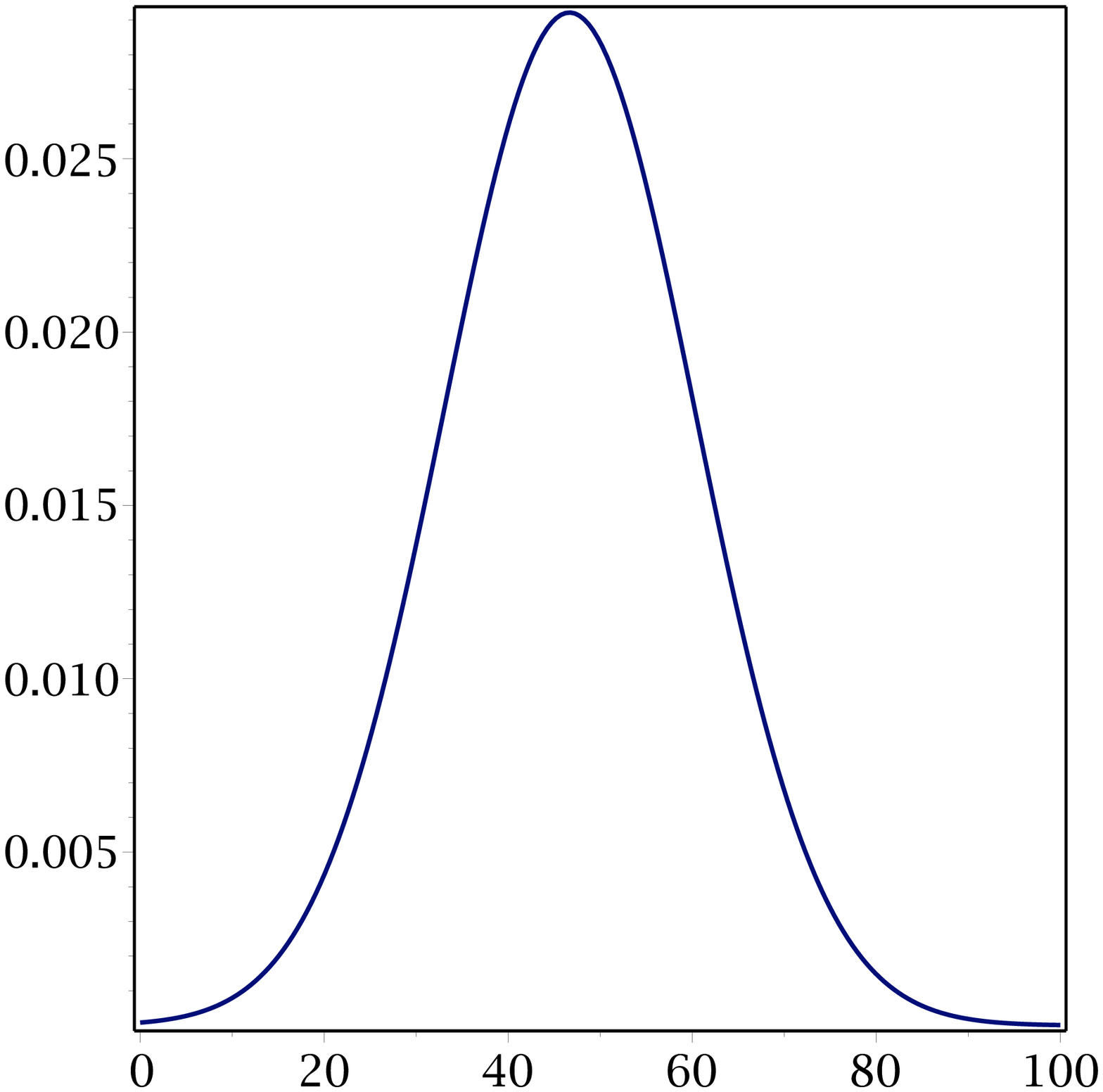}
\caption{On the left: the histogram of the distribution of heights of the points 
in the right-hand picture of Figure \ref{fig:est-by-segm}. 
On the right: the density of the normal distribution with the same mean $46.66$
and standard deviation $13.66$.}
\label{fig:density}
\end{center}
\end{figure}


\subsection{Conclusion}

Our main aim in this note has been to give heuristic and computational evidence that 
there are infinitely many projective primes, especially in the simplest and apparently 
most abundant case, where $n=3$ and $q$ is prime.
We will not pursue these speculations further and will leave the question of more exact 
estimates of the number and distribution of projective primes to the community of experts 
in probabilistic number theory.
(Our own backgrounds and motivation for this investigation lie in the areas of dessins 
d'enfants and permutation groups.)

\bigskip

\paragraph{\bf Acknowledgements} We are greatly indebted to Yuri Bilu who acquainted us 
with the Bunyakov\-sky conjecture, which became a crucial point of our study, and to 
Peter Cameron, Robert Guralnick and Cheryl Praeger for some very useful comments.
Jean B\'etr\'ema helped us with some computations which were too heavy for our Maple
package on a laptop computer. 
Alexander Zvonkin was partially supported by the ANR project {\sc Combin\'e}
(ANR-19-CE48-0011).


\end{document}